\documentclass[11pt]{article}
\usepackage{amsmath}
\usepackage{amsthm}
\usepackage{amssymb}

\usepackage[total={15cm,22.5cm}]{geometry}

\newtheorem{theo}{Theorem}[section]

\newtheorem{prop}[theo]{Proposition}
\newtheorem{lemma}[theo]{Lemma}
\newtheorem{coro}[theo]{Corollary}

\theoremstyle{definition}
\newtheorem*{remark}{Remark}

\newcommand{\FF}{{\cal F}}
\newcommand{\HH}{{\cal H}}

\newcommand{\PP}{{\cal P}}

\newcommand{\eps}{{\varepsilon}}
\DeclareMathOperator{\ex}{ex}
\begin{document}
\date{}

\title{
Largest subgraph from a hereditary property in a random graph
}

\author{Noga Alon\thanks{Department of Mathematics, Princeton University,
Princeton, NJ 08544, USA and
Schools of Mathematics and
Computer Science, Tel Aviv University, Tel Aviv 6997801,
Israel.
Email: {\tt nalon@math.princeton.edu}.
Research supported in part by
NSF grant DMS-2154082 and by USA-Israel BSF grant 2018267.}
\and
Michael Krivelevich\thanks{
  School of Mathematical Sciences,
  Tel Aviv University, Tel Aviv 6997801, Israel.
  Email: \texttt{krivelev@tauex.tau.ac.il}.
  Research supported in part by USA-Israel BSF grant 2018267.}
\and
Wojciech Samotij\thanks{
  School of Mathematical Sciences,
  Tel Aviv University, Tel Aviv 6997801, Israel.
  Email: \texttt{samotij@tauex.tau.ac.il}.
  Research supported in part by the Israel Science Foundation grant 2110/22.}
}

\maketitle
\begin{abstract}
We prove that for every non-trivial hereditary family of graphs $\PP$ and for every
fixed $p \in (0,1)$, the maximum possible number of edges in
a subgraph of
the random graph $G(n,p)$ which belongs to $\PP$ is, with high probability,
$$
\left(1-\frac{1}{k-1}+o(1)\right)p{n \choose 2},
$$
where $k$ is the minimum chromatic number of a graph that does not belong
to $\PP$.
\end{abstract}

\section{Introduction}

Let $\PP$ be an arbitrary hereditary property of graphs. We assume throughout that $\PP$ is non-trivial, i.e., it contains all edgeless graphs and misses some graph.
For a graph $G$, let
$\ex(G,\PP)$ denote the maximum number of edges of a subgraph of
$G$ that belongs to $\PP$; the above definition of non-triviality guarantees that this number is well-defined.
In this note we determine,
for every fixed edge probability $p \in (0,1)$, the
typical asymptotic value of $\ex(G,\PP)$ for the random graph
$G=G(n,p)$ as $n$ tends to infinity. This is stated in the following
theorem.
\begin{theo}
\label{t11}
Let $\PP$ be a non-trivial hereditary property of graphs
and let
$k=k(\PP) \geq 2$ denote the minimum chromatic number of a graph that does
not belong to $\PP$. Then, for every fixed $p \in (0,1)$, the random
graph $G=G(n,p)$ satisfies, with high probability:
$$
\ex(G,\PP)=\left(1-\frac{1}{k-1}+o(1)\right)p {n \choose 2},
$$
where the $o(1)$-term tends to $0$ as $n$ tends to infinity.
\end{theo}

In fact, our argument yields the assertion of the theorem under the 
weaker assumption that
$n^{-c} \le p = p(n) $ and $1-p \geq n^{-c}$
for some positive constant $c$ that depends only on
the property~$\PP$, see Proposition~\ref{prop:sparse-induced-Turan} 
below for the sparse case, the dense case follows in the same way.

In the statement above and throughout the rest of this note, the term
``with high probability" (whp, for short) means, as usual, with
probability tending to $1$ as $n$ tends to infinity.

The assertion of Theorem~\ref{t11} for principal monotone properties
(that is, properties defined by avoiding a single graph)
is known in a strong form,
and the precise range of the probability
$p=p(n)$ for which it holds has been determined in~\cite{CG, Sch},
following a considerable number of earlier results establishing special
cases. There are, however, far more hereditary properties than monotone
ones. A few arbitrarily chosen examples are perfect graphs,
graphs with no induced hole
of length $57$, graphs containing no set of $10$ vertices that span
exactly $34$ edges, or intersection graphs of discs in the plane.

The (short) proof of the theorem is presented in the next section.
The third and final section contains some concluding remarks and open problems.
Throughout the rest of this note, we systematically omit 
all floor and ceiling signs.

\section{The proof}

In this section, we prove Theorem~\ref{t11}.  Let $\PP$ be a hereditary property and put $k=k(\PP)~(\geq 2)$.
Note, first, that every graph $G$ contains a $(k-1)$-colorable subgraph with at least $(1-\frac{1}{k-1}) |E(G)|$ edges.
Since every such subgraph belongs to $\PP$ by the definition of $k=k(\PP)$,
when $G = G(n,p)$ with $p \gg n^{-2}$, then the assertion that
\begin{equation}
  \label{e21}
  \ex(G,\PP) \geq \left(1-\frac{1}{k-1}-
    o(1)\right)p {n \choose 2}
\end{equation}
holds whp follows from the (standard and easy) fact that $|E(G)|$ is concentrated around its expectation
as long as $pn^2$ tends to infinity.

We now turn to the proof of the upper bound on $\ex(G, \PP)$.
We start by observing that, for every graph $G$ and all $L \notin \PP$,
we have $\ex(G, \PP) \le \ex(G, \FF_L)$, where $\FF_L$ is the property
of not containing $L$ as an induced subgraph.  The following proposition
establishes an optimal upper bound on $\ex(G, \FF_L)$ in the case
where $G$ is a \emph{sparse} binomial random graph.  Recall that the
$2$-density of a nonempty graph $L$ is defined by
\[
  m_2(L) = \max\left\{\frac{|E(K)|-1}{|V(K)|-2} : K \subseteq L \text{ with } |E(K)| \ge 2\right\},
\]
when $L$ has at least two edges, and $m_2(K_2) = 1/2$.
This notion is defined so that $p \gg n^{-1/m_2(L)}$ precisely when,
for every subgraph $K \subseteq L$ with at least two edges,
the expected number of copies of $K$ in $G(n,p)$ is asymptotically much bigger than $pn^2$.
In particular, it is not hard to see that the lower-bound assumption $p \gg n^{-1/m_2(L)}$
is necessary.

\begin{prop}
  \label{prop:sparse-induced-Turan}
  Let $L$ be a nonempty graph.  If $n^{-1/m_2(L)} \ll p =p(n) \ll 1$, then the random graph
  $G = G(n,p)$ whp satisfies
  \[
    \ex(G, \FF_L) \le \left(1 - \frac{1}{\chi(L)-1}+o(1)\right) p\binom{n}{2}.
  \]
\end{prop}

We will derive the proposition from the following ``supersaturated'' version
of the random analogue of Tur\'an's theorem in $G(n,p)$ proved by Conlon and Gowers~\cite{CG}
(under an additional technical assumption on $L$) and by Schacht~\cite{Sch} (in full generality):

\begin{theo}[{\cite{CG, Sch}}]
  \label{theo:supersaturated-Turan}
  For every nonempty graph $L$ and every $\eps > 0$, there exists $\delta > 0$
  such that the following holds for every $p = p(n) \gg n^{-1/m_2(L)}$.  With high probability,
  every subgraph $G'$ of the random graph $G(n,p)$ with
  \[
    |E(G')| \ge \left(1-\frac{1}{\chi(L)-1}+\eps\right)p\binom{n}{2}
  \]
  contains at least $\delta n^{|V(L)|} p^{|E(L)|}$ copies of $L$.
\end{theo}

\begin{remark}
  Even though the above result is not explicitly stated in either~\cite{CG} or~\cite{Sch},
  one obtains it easily by: (i)~using the stronger conclusion of~\cite[Theorem~9.4]{CG}
  while deriving~\cite[Theorem~10.9]{CG};  (ii)~replacing~\cite[Theorem~3.3]{Sch} with
  its stronger version~\cite[Lemma~3.4]{Sch} in the derivation of~\cite[Theorem~2.7]{Sch}.
  A stronger version of Theorem~\ref{theo:supersaturated-Turan},
  with optimal dependence of $\delta$ on $\eps$ and $H$, is stated and proved
  in~\cite[Theorem~1.10]{CGSSch}.
\end{remark}

\begin{proof}[Proof of Proposition~\ref{prop:sparse-induced-Turan}]
  It suffices to show that, for every fixed $\eps > 0$, with high probability,
  every subgraph $G' \subseteq G$ with at least $(1-\frac1{\chi(L)-1}+\eps)p\binom{n}{2}$
  edges contains more (not necessarily induced) copies of $L$ than the total
  number of copies of all strict supergraphs of $L$ (with the same vertex set) in $G$.  (We note that a similar idea
  was used in~\cite{CDLFRSch} to derive upper bounds on induced Ramsey numbers.)
  On the one hand, Theorem~\ref{theo:supersaturated-Turan} implies that each $G'$
  as above contains at least $\delta n^{|V(L)|}p^{|E(L)|}$ copies of $L$.  On the other hand,
  a simple application of Markov's inequality gives that, with probability $1-O(\sqrt{p})$, say,
  the total number of copies of all strict supergraphs of $L$ is at most $O(n^{|V(L)|}p^{|E(L)|+1/2})$.
  Since we assume that $p \ll 1$, the latter of the above two quantities is much smaller.
\end{proof}

The following statement is a straightforward corollary of the definition of $k$
and Proposition~\ref{prop:sparse-induced-Turan} invoked with some $L \notin \PP$
with $\chi(L) = k$.

\begin{coro}
  \label{l21}
  For every $\eps>0$, there exists a graph $H$ such that
  \[
    \ex(H, \PP) < \left(1-\frac{1}{k-1}+\eps\right) e(H).
  \]
\end{coro}

We will now deduce the statement of Theorem~\ref{t11} from
Corollary~\ref{l21} and the following standard probabilistic estimate.
\begin{lemma}
\label{l23}
Let $H$ be a fixed graph, let $\eps>0$ be a small positive real and
let $p \in (0,1)$ be a fixed real. Let $G=G(n,p)$ be the random graph
and let $c=c(H,n,p)$ be the expected number of induced copies of $H$
in $G$ that contain a fixed edge of $G$. Then, whp, for every edge
$e$ of $G$
the number of induced copies of $H$ in $G$ containing $e$ is at least
$(1-\eps)\mu$ and at most $(1+\eps)\mu$.
\end{lemma}
\begin{proof}
  Fix an edge $e$ of $K_n$, assume it belongs to $E(G)$,
  and apply the edge exposure martingale to the
  random variable $X_e$ counting the number of induced copies of $H$ in $G$
  that contain $e$. As $p$ and $H$ are fixed, the expectation of this
random variable is $\mu=\Theta(n^{h-2})$, where $h$ is the number of vertices of
$H$ and the hidden constant in the $\Theta$-notation is a function of $p$ and
$H$. The existence or nonexistence of each of the potential
$2n-4$ edges of $G$ (besides $e$) incident with $e$ can change the
value of $X_e$ by at most $O(n^{h-3})$. Similarly, each of the
${{n-2} \choose 2}$ other edges can change the value of $X_e$ by at most
$O(n^{h-4})$. Therefore, by Azuma's Inequality (see, e.g., \cite[Chapter~7]{AS}),
the probability that $X_e$ deviates from its expectation by
$$
\lambda \cdot [ (2n-4) O(n^{2h-6})+{{n-2} \choose 2} O(n^{2h-8})]^{1/2}
=\lambda \cdot O(n^{h-5/2})
$$
is at most $2e^{-\lambda^2/2}$.
In particular, the probability that $X_e$ deviates from its expectation
$\mu$ by $a \cdot \mu \frac{\sqrt{\log n}}{\sqrt n}$ is at most
$n^{-\Omega(a^2)}$. Taking a sufficiently large constant $a$,  this
implies the desired result by the union bound.
\end{proof}

\begin{proof}[Proof of Theorem~\ref{t11}]
  Let $\PP$ and $k=k(\PP)$ be as in the statement of the theorem.
  Fix a small real $\eps>0$ and let $H$
  be the graph from Corollary~\ref{l21}.  Fix $p \in(0,1)$
  and let $G=G(n,p)$ be the binomial random graph.
  Let $\HH=\{H_i\}_{i \in I}$
  be the collection of all induced copies of $H$ in $G$.
  By Lemma~\ref{l23}, whp the number of subgraphs $H_i$ containing
any edge of $G$ is at least $(1-\eps)\mu$ and at most $(1+\eps)\mu$, where $\mu>0$
is the expected number of such subgraphs. Assuming this holds, let
$G'$ be a subgraph of $G$ that belongs to $\PP$.  We need
to estimate the number of edges of $G'$ from above.  We will do so by comparing
two bounds on the size of the set $S$ of all ordered pairs  $(e,H_i)$, where $e$ is an edge
of $G'$, $H_i \in \HH$, and $e$ is also an edge of $H_i$. Since every
edge of $G'$ is contained in at least $(1-\eps)\mu$ of the graphs $H_i$,
 we have:
\begin{equation}
\label{e22}
|S| \geq |E(G')| (1-\eps)\mu.
\end{equation}
On the other hand, since $H$ satisfies the assertion of Lemma~\ref{l21},
\begin{equation}
\label{e23}
|S| \leq |\HH| \left(1-\frac{1}{k-1}+\eps\right)|E(H)|
\end{equation}
as every $H_i$ can contain at most $(1-\frac{1}{k-1}+\eps)|E(H)|$
edges of $G'$. Indeed,
$$(V(H_i),E(G') \cap E(H_i))
$$
is an induced subgraph
of $G'$, and as $G'$ lies in
$\PP$ which is hereditary, so does this graph. Therefore, each $H_i \in \HH$
can contain  at most $(1-\frac{1}{k-1}+\eps)|E(H)|$ edges of $G'$. Finally,
since no edge of $G$ lies in more than $(1+\eps)\mu$ members of
$\HH$, it follows that
\begin{equation}
\label{e24}
|\HH| |E(H)| \leq |E(G)| (1+\eps)\mu.
\end{equation}

Combining \eqref{e22}, \eqref{e23}, and \eqref{e24}, we conclude that
\begin{multline*}
  |E(G')| (1-\eps) \mu \leq |S| \leq  |\HH| \left(1-\frac{1}{k-1}+\eps\right)|E(H)| \\
  \leq \frac{|E(G)| (1+\eps)\mu}{|E(H)|} \left(1-\frac{1}{k-1}+\eps\right)|E(H)|
  = \left(1-\frac{1}{k-1}+\eps\right) |E(G)| (1+\eps)\mu.
\end{multline*}
Therefore
\[
|E(G')| \leq \left(1-\frac{1}{k-1}+\eps\right) |E(G)| \frac{1+\eps}{1-\eps}
<\left(1-\frac{1}{k-1}+4 \eps\right) |E(G)|,
\]
where here we used that $\eps>0$ is small. Since $\eps$ can be chosen
to be arbitrarily small, this shows that whp
$$
\ex(G,\PP) \leq \left(1-\frac{1}{k-1}+o(1)\right) |E(G)|
=\left(1-\frac{1}{k-1}+o(1)\right) p {n \choose 2}.
$$
This, together with~\eqref{e21}, completes the proof of the theorem.
\end{proof}

\section{Concluding remarks}
\begin{itemize}
\item
Theorem~\ref{t11} can be proved in a more self-contained way, using
Szemer\'edi's Regularity Lemma and following
the approach in~\cite[Section~4.4]{AGKMS}. Since unlike the
proof described here, this proof does not work
for $p \leq n^{-\eps}$ when $\eps>0$ is fixed,
we omit the details.
\item
  It may be interesting to characterize all edge probabilities $p$ for which
  the assertion of Theorem~\ref{t11} holds. It is not difficult to describe hereditary
  properties (even monotone ones) for which
  the fraction of edges of $G(n,p)$ that lie in a maximum subgraph
  that belongs to the property changes, whp, several times
  as $p=p(n)$ increases from $0$ to $1$.
\item
  Our main result, Theorem~\ref{t11}, shows that if $\PP$ misses a bipartite graph,
then for $G=G(n,p)$ we have whp $\ex(G,\PP)=o(n^2)$. It would be interesting
to provide a more accurate estimate for this ``bipartite" case.
It seems plausible that in this case
$\ex(G,\PP)\le n^{2-\eps}$ for some $\eps=\eps(\PP)>0$.
\item
The typical edit distance of a random graph from a hereditary property
is very different from the typical minimum number of edges that have to be
deleted from it to get a graph that belongs to the property. The latter
quantity is the one studied here, the former is treated, for example,
in~\cite{ASt}. An example illustrating the difference is that of
the property of avoiding an induced copy
of a long even cycle~$C_{2k}$.
Theorem~\ref{t11} shows that for, say, $p=1/2$, whp one has to delete
nearly all $(1/4+o(1))n^2$
edges of $G(n,1/2)$ to get a subgraph that contains no
induced copy of $C_{2k}$. On the other hand, since the vertices of
$C_{2k}$ cannot be covered by $k-1$ cliques it suffices, whp, to
add to $G(n,1/2)$ only $(1/4+o(1))n^2/(k-1)$ edges in order to
cover all its vertices by $k-1$ cliques, ensuring that the resulting
graph will not contain an induced copy of $C_{2k}$.
\end{itemize}

\section*{Acknowledgment}
The initial results in this note were obtained when the
second author visited the Department of Mathematics of Princeton University.
He would like to thank the department for the hospitality.

\bibliographystyle{amsplain}
\bibliography{exher}

\end{document}